\documentclass[12pt]{amsart}

\usepackage{amsfonts , graphicx, amsmath, amssymb, color}
 \usepackage[applemac]{inputenc}

\usepackage[all, 2cell]{xypic}
\usepackage[margin = 1in]{geometry}

\newcommand \pa {\partial}

\newtheorem{theorem}{Theorem}
\newtheorem{corollary}[theorem]{Corollary}

\newtheorem{lemma}[theorem]{Lemma}

\newtheorem{definition}[theorem]{Definition}

\newcommand\ind{\operatorname{ind}}
\newcommand\sgn{\operatorname{sgn}}

\newcommand\IH{\operatorname{IH}}
\newcommand\WH{\operatorname{WH}}
\newcommand\Hom{\operatorname{Hom}}
\newcommand\bM{\overline{M}}
\newcommand\bU{\overline{\mathcal{U}}}
\newcommand\tM{\widetilde{M}}
\newcommand\bL{\overline{L}}
\newcommand\bW{\overline{W}}
\newcommand\bcU{\overline{\mathcal{U}}}
\newcommand\bm{\overline{m}}
\newcommand\bbN{\mathbb{N}}

\newcommand\cU{\mathcal{U}}

\newcommand\bbR{\mathbb{R}}

\newcommand\bbZ{\mathbb{Z}}

\newcommand\CI{\mathcal{C}^{\infty}}
\newcommand\cC{\mathcal{C}}

\newcommand\w{{\rm w}}

\title[Hodge cohomology of iterated fibred cups metrics]{Weighted Hodge cohomology of iterated fibred cusp metrics}
 
\author{Eug\'enie Hunsicker}
\address{Department of Mathematical Sciences, Loughborough University}
\email{E.Hunsicker@lboro.ac.uk}
\author{Fr\'ed\'eric Rochon}
\address{Département de Mathématiques, UQÀM}
\email{rochon.frederic@uqam.ca}

\begin{document}

\maketitle

\begin{abstract}
On a smoothly stratified space, we identify intersection cohomology of any given perversity with 
an associated weighted $L^2$ cohomology for iterated fibred cusp metrics on the smooth stratum.
\end{abstract}
\bigskip

The Hodge theorem for smooth compact manifolds establishes an important link between two analytic
invariants of a manifold, the vector space of ($L^2$) harmonic forms over the manifold and the 
($L^2$) cohomology, and a topological 
invariant of the manifold, the cohomology with real coefficients, calculated using cellular, simplicial
or smooth deRham theory.  In the 1980's Cheeger, together with Goresky and MacPherson \cite{C1980}, \cite{CGM1982}, discovered that
this important link extends to a subset of pseudomanifolds called Witt spaces, where the relationship
is between $L^2$ cohomology with respect to an (incomplete) iterated conical metric and the (unique) 
middle perversity intersection cohomology, again with real coefficients, see also \cite{Nagase1987}, \cite{BHS1992} and \cite{ALMP2011}. 

The purpose of this paper is to show that by considering instead a natural class of complete metrics called iterated fibred cusp metrics, one can circumvent the analytical difficulties due to the incompleteness of iterated conical metrics, opening in this way a simpler and more direct route towards  a description of intersection cohomology in terms of harmonic forms.  More precisely, for any smoothly stratified space, $X$, we obtain a Hodge theorem identifying the intersection cohomology of $X$ of a given perversity with a weighted $L^2$ cohomology of the 
regular set $X \setminus X_{\rm sing}$, endowed with a complete iterated fibred cusp metric
which near the singular strata has hyperbolic cusp type behaviour on the link.
The Kodaira decomposition theorem
tells us that when the weighted $L^2$ cohomology with respect to some metric is finite dimensional, it is 
isomorphic to the space of weighted $L^2$ harmonic forms, so as a consequence, we also get
an isomorphism to the space of weighted $L^2$ harmonic forms for this metric.  By reference
to the theory of Hilbert complexes, see \cite{BL1992}, we additionally obtain the corollary
that the weighted Hodge Laplacian associated to such a metric and weight is Fredholm as an 
unbounded operator on the space of weighted $L^2$ differential forms.  In particular, when $X$ is a Witt space,
we show that the unweighted $L^2$ cohomology on $X$ with respect to a complete iterated fibred
cusp metric is isomorphic to the (unique) middle perversity intersection cohomology of $X$.
 
Our Hodge theorem can be seen as a natural generalization, from the setting of smooth K\"ahler 
manifolds to the singular Witt setting, the result by Timmerscheidt in which the 
$L^2$ cohomology of the complement of a normal crossings divisor in a K\"ahler manifold, endowed with a particular complete
metric, is shown to be isomorphic to the cohomology of its (smooth) compactification (contained in an appendix to \cite{EVT}).  
Our results also have an intersection with the Zucker conjecture, proved by Saper and Stern, and independently by Looijenga, 
which says that the $L^2$ cohomology of a Hermitian locally symmetric space is isomorphic to the middle perversity intersection cohomology of its
Borel-Bailey compactification \cite{SS1990}.  All of these results have in common that the metric in question is complete, and near each 
compactifying stratum, has some generalised hyperbolic cusp type behaviour.  Furthermore, in our theorem, as in the result of Timmerscheidt and 
in the Zucker conjecture, the $L^2$ cohomology is finite dimensional due to the vanishing of some cohomology classes that arise in 
its calculation.  For general iterated fibred cusp metrics, without the Witt condition, the $L^2$ cohomology will be infinite dimensional.  In this
case, harder analytic results are required to prove a Hodge theorem.  This was carried out in the case of one singular stratum in 
\cite{HHM2004},  where the first author and her collaborators proved that the space of $L^2$ harmonic forms for a manifold with fibration boundary, endowed with a fibred cusp metric (the base case for iterated fibred cusp metrics), is
isomorphic to the image of lower middle perversity intersection cohomology in upper middle perversity
intersection cohomology of its fibrewise compactification, see also \cite{GR} for a recent generalization of this result to some manifolds with foliated boundary.  

First let us recall what is meant by weighted $L^2$ harmonic forms and the weighted $L^2$ cohomology over
a manifold $M$ with respect to a metric $g$ and a weight $\w$.
A weighted $L^2$ space for any metric $g$ and positive weight function $\w$ on $M$
is a space of forms:
\[
L^2\Omega^*(M,g,{\rm w}) := \{ \omega \in \Omega_{}^*(M) \mid
\int_M ||\w^{-1} \omega||_{g}^2 {\rm dvol}_{g} < \infty\}.
\]
Here $||\,  \cdot \, ||_{g}$ is the pointwise metric on the space of differential forms over $M$
induced by the metric on $M$.  This may be completed to a Hilbert space with respect to the inner product
\[
\langle \omega, \eta \rangle_{g,\w} = \langle \w^{-1} \omega,\w^{-1}\eta \rangle_{g},
\]
where the inner product on the right is the standard $L^2$ inner product on differential forms over $M$ for 
the metric $g$.

Let $d$ represent the de Rham differential 
on smooth forms over $M$ and $\delta_{g,\w}$ represent its formal adjoint with respect to $\langle \cdot,\cdot\rangle_{g,\w}$.  Then $D_{g,\w}:= d + \delta_{g,\w}$ is an elliptic differential operator on 
the space of smooth forms over $M$.  If $\w=1$, the elements of the kernel of $D_{g,\w}$ that lie in $L^2$ are the standard 
space of $L^2$ harmonic forms over $(M,g)$.  More generally, we denote the space of weighted $L^2$ harmonic
forms on $(M,g)$ by:
\[
\mathcal{H}^*_{L^2}(M, g,\w):= \left\{\omega \in  L^2\Omega^*(M,g,\w) \mid D_{g,\w}\omega =0 \right\}.
\] 

The weighted $L^2$ cohomology of a complete manifold
may be computed from its complex of weighted $L^2$ forms:  
$$
\xymatrix{
   \cdots \ar[r]^-{d} & L^2_d\Omega^{k-1}(M,g,\w) \ar[r]^d & L^2_d\Omega^{k}(M,g,\w) \ar[r]^d & L^2_d\Omega^{k+1}(M,g,\w) \ar[r]^-{d} & \cdots,
      }
$$
where 
$$
L^2_d\Omega^{k-1}(M,g,\w) = \{ \omega\in L^2\Omega^k(M,g,\w) \; ; \; d\omega \in L^2\Omega^{k+1}(M,g,\w)   \}.$$
That is, the $L^2$-cohomology is given by
$$
  L^2H^k(M,g,\w) := \{ \omega\in L^2\Omega^k(M,g,\w) \; ;\; d\omega=0\} /  \{d\eta \; ; \; \eta\in L^2_d\Omega^k(M,g,\w)\}.
$$
It is standard that the space $L^2\Omega^*(M,g,{\rm w})$ does not depend sensitively on $g$ and $\w$, but
only on $g$ up to quasi-isometry, where $g$ and $g'$ are quasi-isometric if for all $p\in M$ and $v,w \in T_pM$,
\[
\frac{1}{c} g_p(v,w) \leq g'_p(v,w) \leq c g_p(v,w)
\]
for some uniform constant $c>0$.  Similarly, two weight functions give the same weighted $L^2$ space if they
are equivalent weights in the sense that
\[
\frac{1}{c}\w(p)\leq \w'(p) \leq c \w(p)
\]
for all $p \in M$ and some uniform constant $c>0$.  Since the operator $d$ does not depend on either $g$ or $\w$,
this implies in turn that $L^2H^k(M,g,\w)$ only depends on $g$ up to quasi-isometry and on $\w$ up to equivalence
of weight functions.

Next let us recall some facts about smoothly stratified spaces and their resolutions as manifolds with corners.  Let $X$ be a smoothly stratified space in the sense of \cite{ALMP2011}, see also \cite{verona} and \cite{DLR2011}.  
There is a resolution of $X$ by a manifold with fibred corners, $q:\tM \to X$, where $q$ is a diffeomorphism from the interior of $M$ to the regular stratum of $X$, and is a fiber bundle $q_i: H_i \to S_i$ over some singular
stratum when restricted to the interior of each boundary hypersurface of $M$.  The fiber, $F_i$, of $q_i$ is in turn a resolution of the link $\bL_i$ of the stratum $S_i \subset X$.  

Recall that the (open) singular strata, $S_i$ of $X$ satisfy a partial order relationship, whereby we write 
$S_i \leq S_j$ if $S_i \subset \overline{S}_j$, and this partial order lifts to a partial order on boundary hypersurfaces.
The depth of a stratum $S_i$ is the maximum number of strata that lie above it in the partial order relationship.
The link of any stratum in $X$ has strictly lower depth than $X$, which permits recursive arguments on depth.

There exists on $\tM$ a collection of positive functions $x_i$ that define the hypersurfaces, $H_i$ of $\tM$ and descend to a collection
of functions on $X$ (which we refer to by the same notation) that define the singular strata $S_i$.  These may be chosen so that 
for each $p\in S_i$, there exists a neighborhood $\bU \subset X$ of $p$ and a smoothly stratified map $\phi_{\bU}$ of the form
$$
     \phi_{\bU}: \bU \cong C_1(\bL_i)\times V,
$$
where $V$ is an open ball in the Euclidean space whose image under the homeomorphism is compactly contained in $S_i$, 
$\bL_i$ is a smoothly stratified space, and
$$
     C_1(\bL_i)= [0,1)_{x'_i}\times \bL_i/ \{0\}\times \bL_i
$$
is the cone over $\bL_i$.  
We say $\bU$ is a \textbf{regular neighborhood} of $p$ in $X$.  Such an associated
set of stratum defining functions and neighborhoods is called a set of iteratively defined ``control data".  
The space $X$ may be covered by a finite number of such neighbourhoods.
On such a neighborhood, $x'_i=x_i\circ \phi_{\bU}^{-1}$.  Furthermore, by \cite[Lemma~1.4]{DLR2011}, we can assume that each other stratum defining
function $x_j$ is  in this neighborhood factorizes through $\phi_{\bU}$ and projection onto a stratum defining function on $\bL_i$ (constant on $V$) if $S_i\subset \overline{S}_j$ or a positive function uniformly bounded away from zero otherwise.  Those $x_j$ factorizing through $\bL_i$ then form corresponding stratum defining functions on $\bL_i$ (see \cite{DLR2011}).

We can now define the metrics we will consider on $M:=X \setminus X_{\rm sing} = \tM \setminus \pa\tM$.

\begin{definition}
A \textbf{quasi iterated fibred cusp metric} is defined inductively on the depth of $X$ to be a complete Riemannian metric $g$ on the regular set $M$ of $X$ such that for any $p\in X_{\rm sing}$ and any regular neighborhood $\bU=C_1(\bL)\times V$ of $p$, the restriction of $g$ to $\cU=\bU\setminus (\bU \cap X_{\rm sing})$ is 
quasi isometric to a metric of the form
\begin{equation}\label{eq1}
        g_{\cU}=  \frac{dx^2}{x^2}+ g_{V} + x^2g_L,
\end{equation}
where $g_V$ is a Riemannian metric on $V$ and $g_L$ is a quasi iterated fibred cusp metric on the interior $L$ of $\bL$. 
\label{qifc.1}\end{definition} 

We make two notes about these metrics.  First of all, if we rearrange and change coordinates in (\ref{eq1}), letting $r=-\log(x) \in (\epsilon, \infty)$, we find that
the metric takes on the more familiar generalised hyperbolic cusp form on fibres:
$$
        g_{\cU}=  dr^2 + e^{-2r}g_L+ g_{V}.
$$
Second, we note that the iterated fibred cusp metrics of \cite{DLR2011} constitute a special case of quasi iterated fibred cusp metrics.  In fact, quasi iterated fibred cusp metrics could alternatively be defined as complete Riemannian metrics on the regular set $M$ of $X$ that are quasi-isometric to iterated fibred cusp metrics.  Since the  
 the $L^2$ cohomology on a Riemannian manifold depends only on the quasi-isometry class of the metric, considering this larger class of metrics leads to no extra difficulties.  This also means that, when calculating the $L^2$ cohomology on a regular neighborhood $U$, we are free to consider a model metric as in Equation (\ref{eq1}).

We may now state our result.

\begin{theorem}
Let $g$ be a quasi iterated fibred cusp metric on the interior $M$ of a smoothly stratified space, $X$, and let $\{x_i\}_{i=1}^n$ denote the set of boundary 
defining functions on the resolution $\widetilde{M}$ of $X$ given by a choice of iterated control data.  Let $\frak{p}$ be a perversity,
$0<\epsilon<\frac12$,
and define the weight function $\w_\frak{p}=x_1^{i_1} \cdots x_n^{i_n}$, where $i_j = \frak{p}(c_j)  - (c_j-3)/2+ \epsilon$ 
when the stratum $S_j$ corresponding to the function $x_j$ has codimension $c_j$.
Then the weighted $L^2$ cohomology of $(M,g)$
with respect to the weight function $\w_\frak{p}$ and the space of $\w_\frak{p}$-weighted $L^2$ harmonic forms are naturally isomorphic to the intersection cohomology of $X$ of perversity $\frak{p}$:
$$
    \mathcal{H}^*_{L^2}(M,g,\w_\frak{p})\cong L^2H^*(M,g,\w_\frak{p}) \cong \IH^*_{\frak{p}}(X).
$$ 
Moreover, $d$ and its formal adjoint have closed range.
\label{ifc.1}\end{theorem}

Before we begin the proof, we recall that 
we can turn the weighted $L^2$ complex of forms over $M$ into a complex of sheaves over $X$ as follows.  
Over each regular neighborhood $\bU$ in $X$, we define the complex of weighted $L^2$-forms
associated to the iterated fibred cusp metric $g$ by:
$$
\xymatrix{
   \cdots \ar[r]^{d \hspace{1cm}} & L^2_d\Omega^{k-1}(\bU,g,\w) \ar[r]^d & L^2_d\Omega^{k}(\bU,g,\w) \ar[r]^{d \hspace{0.3cm}} & L^2_d\Omega^{k+1}(\bU,g,\w) \ar[r]^{\hspace{0.8cm} d} & \cdots
      },
$$
where
$$
     L^2_d\Omega^{k-1}(\bU,g,\w) = \{ \omega\in L^2\Omega^k(\cU,g,\w) \; ; \; d\omega \in L^2\Omega^{k+1}(\cU,g,\w)   \}
$$
and $\cU= \bU\cap M$.
This gives a complex of presheaves on $X$,
$$
\xymatrix{
   \cdots \ar[r]^{d \hspace{0.5cm}}& \w L^2_d\Omega^{k-1}\ar[r]^d & \w L^2_d\Omega^{k}\ar[r]^d & \w L^2_d\Omega^{k+1} \ar[r]^d & \cdots.
      }
$$
These pre-sheaves may be sheafified to form the complex of sheaves that we also denote $\w L^2_d\Omega^k$ on $X$.
Before proving our main theorem, the following result will be useful.
\begin{lemma}
For all $k\in \bbN_0$, the sheaf $\w L^2_d\Omega^k$ on $X$ is fine.  
\label{qifc.2}\end{lemma}
\begin{proof} 
We need to show that the sheaf $\Hom_{\bbZ}(\w L^2_d\Omega^k, \w L^2_d\Omega^k)$ admits partitions of unity.  Let $\{\bU_i\}$ be an open cover of $X$ by open sets in $M$ and regular neighborhoods in $X$.  Since $X$ is compact, we can assume that this cover contains finitely many open sets.    By slightly shrinking each of the $\bU_i$, we can get another cover $\{\bW_i\}$ of the same form, where $\bW_i\subset \bU_i$ for each $i$.  For each $\bU_i$ contained in $M$, we can then consider a smooth nonnegative function $\rho_i\in \CI_c(\bU_i)$ such that $\rho_i$ is nowhere zero on $\bW_i\subset \bU_i$.  When $\bU_i$ is a regular neighborhood of a point $p_i\in\pa\bM$ of the form
$$
        \bU_i = C_1(\bL_i)\times V_i,
$$  
we can take a nonnegative function $\rho_i\in \cC^0_c(\bU_i)$ with $\rho_i$ nowhere zero on $\bW_i$ to be of the form $\rho_i= \pi^*_i \phi_i$ for some $\phi_i\in \CI_c([0,1)\times V_i)$, where $\pi_i: C_1(\bL)\times V_i\to [0,1)_x\times V_i$ is the obvious projection.  In this way, we have that $d\rho_i\in L^{\infty}\Omega^1(M,g)$.  This means the functions $\psi_i = \frac{\rho_i}{\sum_j \rho_j}$ form a partition of unity for the sheaf $\Hom_{\bbZ}(\w L^2_d\Omega^k, \w L^2_d\Omega^k)$.

\end{proof}

We can  now prove Theorem \ref{ifc.1}.
\begin{proof}
We proceed by induction on the depth of $X$.  If $X$ has depth zero, the weight function may be taken as constant and the result follows from the standard Hodge theorem on compact manifolds, since all positive weights give the same weighted cohomology.  Suppose now that $X$ has depth $d$ and that the theorem holds for all spaces of depth less than $d$ and perversities on these spaces.  Given a point 
$p$ in $S \subset X\setminus M$, where $S$ has codimension $c$, we can find a regular neighborhood $\bcU$ of $X$ of the form
$$
     \bcU= C_1(\bL)\times V,
$$
where $V$ is an open ball in the Euclidean space and $\bL$ is a smoothly stratified space of depth less than $d$.  If $g_L$ is an iterated fibred cusp metric on the interior $L$ of $\bL$ and $g_V$ is a choice of Riemannian metric on $V$, then the metric
$$
      g_{\cU}= \frac{dx^2}{x^2}+ g_V + x^2 g_L
$$
is an iterated fibered cusp metric on $\cU=\bcU\setminus(\cU\cap X_{\rm sing})$.  In particular, $g_{\cU}$ is quasi-isometric to the restriction of $g$ on $\cU$, so we can use $g_{\cU}$ instead of $g$ to compute the $L^2$-cohomology of $\cU$.  Further, the weight function in this neighbourhood is equivalent to one of the form
$\w_{\frak{p}}(x,v,s)=x^{i_j} \w'_{\frak{p}}(s)$, where $\w'(s)$ is a weight function on the link, $L$, corresponding to the same perversity.  

Now, by our induction hypothesis, the exterior differential $d$ on $(L,g_L)$ has closed range, so we can apply the Kunn\"eth formula of Zucker \cite[Corollary~2.34]{Zucker} to the metric $g_{\cU}$, which gives,
$$
     L^2H^k(\cU,g,\w)= \bigoplus_{i=0}^1\WH^i((0,1),\frac{dx^2}{x^2},k-i-\frac{\ell}{2}+i_j)\otimes L^2H^{k-i}(L,g_L,\w'), $$
where $\ell=\dim L$  and 
$$
\WH^i((0,1),\frac{dx^2}{x^2},a)= \frac{\{ \omega\in x^aL^2\Omega^i((0,1),\frac{dx^2}{x^2}) \; | \; d\omega=0\}}{\{d\eta\; | \; \eta\in x^aL^2\Omega^{i-1}((0,1),\frac{dx^2}{x^2})\;d\eta\in x^aL^2\Omega^i((0,1),\frac{dx^2}{x^2})\}}
$$ is the weighted $L^2$-cohomology on the interval $(0,1)$ with metric $\frac{dx^2}{x^2}$ and weight $x^a$ as defined in \cite[(12)]{HHM2004}.     Now, we have that $WH^1((0,1),\frac{dx^2}{x^2},a)=\{0\}$ for $a\ne 0$ and is infinite dimensional when $a=0$ (which can never happen by our choice of $\w$), while
$$
      \WH^0((0,1),\frac{dx^2}{x^2},a)=  \left\{ \begin{array}{ll}
            \{0\}, & a\ge 0, \\
            \bbR, & a<0.
      \end{array}
        \right.
$$ 
On the other hand, by our induction hypothesis, $L^2H^*(L,g_L,\w')\cong \IH_{\frak{p}}^*(\bL)$.  Thus we obtain that
$$
  L^2H^k(\cU)= \left\{  \begin{array}{ll}
  \IH^k_{\frak{p}}(L), & k< c_j-2-\frak{p}(c_j) - \epsilon, \\
      \{0\}, & k\ge c_j-2-\frak{p}(c_j) - \epsilon.
  \end{array}
        \right.
$$
Because $k$ is an integer, this can be rewritten as
$$
  L^2H^k(\cU)= \left\{  \begin{array}{ll}
  \IH^k_{\frak{p}}(L), & k < c_j-2 -\frak{p}(c_j), \\
      \{0\}, & k \geq c_j-2-\frak{p}(c_j).
  \end{array}
        \right. 
$$
Since $p\in X \setminus M$ was arbitrary, we can conclude by  \cite[Proposition~1]{HHM2004} that there is a natural isomorphism
$$
     L^2H^*(M,g,\w)\cong \IH^*_{\frak{p}}(X).
$$
Since $\IH^*_{\frak{p}}(X)$ is finite dimensional, this means that, the ranges of $d$ and its formal
adjoint $\delta_{g,\w}$ are closed.  From the Kodaira decomposition theorem 
$$
  L^2 \Omega^k (M,g)= \mathcal{H}_{L^2}^{k} (M,g,\w)\oplus \overline{d \mathcal{C}_c^{\infty}(M; \Lambda^{k-1}(T^*M))} \oplus \overline{\delta_{g,\w} \mathcal{C}^{\infty}_c(M; \Lambda^{k+1}(T^*M))},
$$
we therefore conclude as in \cite[\S~2.1]{HHM2004} that  the space of $L^2$ harmonic forms $\mathcal{H}_{L^2}^*(M,g,\w)$ is naturally identified with $L^2H^*(M,g,\w)$.
\end{proof}

In particular, if $X$ is a Witt space, we can consider unweighted $L^2$ cohomology and get:
\begin{theorem}
Let $g$ be a quasi iterated fibred cusp metric on the interior $M$ of a smoothly stratified Witt space, $X$.
Then the $L^2$ cohomology of $(M,g)$
and the space of $L^2$ harmonic forms are naturally isomorphic to the middle perversity intersection cohomology of $X$:
$$
    \mathcal{H}^*_{L^2}(M,g)\cong L^2H^*(M,g) \cong \IH^*_{\bm}(X).
$$ 
Moreover, $d$ and its formal adjoint have closed range.
\end{theorem}
\begin{proof}
The proof is identical to the one above, except we need to check that the unweighted cohomology is never infinite dimensional.
This would occur if we get the weight $a=0$ arising in a term of the form
\[
\WH^1((0,1),\frac{dx^2}{x^2},k-i-\frac{\ell}{2})\otimes L^2H^{k-i}(L,g_L),
\]
where the second term in the product is nontrivial.  However, by our induction hypothesis, $L^2H^*(L,g_L)\cong \IH_{\bm}^*(\bL)$.  In particular, the Witt condition then implies that $L^2H^{\frac{\ell}{2}}(L,g_L)=\{0\}$ if $\ell$ is even, so this can never occur.
\end{proof}
This easily leads to the following consequences for the signature operator.

\begin{corollary}
Let $g$ be a quasi iterated fibred cusp metric on the interior $M$ of a smoothly stratified Witt space $X$.  Then the associated signature operator $D_{g}=d+\delta_{g}$ is Fredholm with index given by 
$$
    \ind(D)=\sgn( \IH^*_{\bm}(X)).
$$
\label{ifc.2}\end{corollary}
\begin{proof}
This follows from the fact that the sheaf $L^2_d\Omega^k$ is a locally self-dual sheaf, and the
axiomatic characterisation of intersection cohomology.
\end{proof}

%
%
%
%
%
 \bibliography{hodgeifc}

\providecommand{\bysame}{\leavevmode\hbox to3em{\hrulefill}\thinspace}
\providecommand{\MR}{\relax\ifhmode\unskip\space\fi MR }
\providecommand{\MRhref}[2]{%
  \href{http://www.ams.org/mathscinet-getitem?mr=#1}{#2}
}
\providecommand{\href}[2]{#2}
\begin{thebibliography}{10}

\bibitem{ALMP2011}
P.~Albin, E.~Leichtnam, R.~Mazzeo, and P.~Piazza, \emph{The signature package
  on {W}itt spaces}, Annales Scientifiques de l'\'Ecole normale Sup\'erieure
  \textbf{45} (2012).

\bibitem{BHS1992}
Jean-Paul Brasselet, Gilbert Hector, and Martin Saralegi,
  \emph{{${\mathcal{L}}^2$}-cohomologie des espaces stratifi\'es}, Manuscripta
  Math. \textbf{76} (1992), no.~1, 21--32. \MR{1171153 (93i:58009)}

\bibitem{BL1992}
J.~Br{\"u}ning and M.~Lesch, \emph{Hilbert complexes}, J. Funct. Anal.
  \textbf{108} (1992), no.~1, 88--132.

\bibitem{C1980}
Jeff Cheeger, \emph{On the {H}odge theory of {R}iemannian pseudomanifolds},
  Geometry of the {L}aplace operator ({P}roc. {S}ympos. {P}ure {M}ath., {U}niv.
  {H}awaii, {H}onolulu, {H}awaii, 1979), Proc. Sympos. Pure Math., XXXVI, Amer.
  Math. Soc., Providence, R.I., 1980, pp.~91--146.

\bibitem{CGM1982}
Jeff Cheeger, Mark Goresky, and Robert MacPherson, \emph{{$L^{2}$}-cohomology
  and intersection homology of singular algebraic varieties}, Seminar on
  {D}ifferential {G}eometry, Ann. of Math. Stud., vol. 102, Princeton Univ.
  Press, Princeton, N.J., 1982, pp.~303--340.

\bibitem{DLR2011}
C.~Debord, J.-M. Lescure, and F.~Rochon, \emph{Pseudodifferential operators on
  manifolds with fibred corners}, arXiv:1112.4575, to appear in Ann. Inst.
  Fourier, 2011.

\bibitem{EVT}
H{\'e}l{\`e}ne Esnault and Eckart Viehweg, \emph{Logarithmic de {R}ham
  complexes and vanishing theorems}, Invent. Math. \textbf{86} (1986), no.~1,
  161--194.

\bibitem{GR}
J.~Gell-Redman and F.~Rochon, \emph{Hodge cohomology of some foliated boundary
  and foliated cusp metrics}, Mathematische Nachrichten \textbf{19} (2015),
  no.~3, 719--729.

\bibitem{HHM2004}
Tam{\'a}s Hausel, Eugenie Hunsicker, and Rafe Mazzeo, \emph{Hodge cohomology of
  gravitational instantons}, Duke Math. J. \textbf{122} (2004), no.~3,
  485--548.

\bibitem{Nagase1987}
Masayoshi Nagase, \emph{Sheaf theoretic {$L^2$}-cohomology}, Complex analytic
  singularities, Adv. Stud. Pure Math., vol.~8, North-Holland, Amsterdam, 1987,
  pp.~273--279. \MR{894298 (88g:58009)}

\bibitem{SS1990}
Leslie Saper and Mark Stern, \emph{{$L_2$}-cohomology of arithmetic varieties},
  Ann. of Math. (2) \textbf{132} (1990), no.~1, 1--69.

\bibitem{verona}
Andrei Verona, \emph{Stratified mappings---structure and triangulability},
  Lecture Notes in Mathematics, vol. 1102, Springer-Verlag, Berlin, 1984.
  \MR{771120 (86k:58010)}

\bibitem{Zucker}
Steven Zucker, \emph{{$L_{2}$} cohomology of warped products and arithmetic
  groups}, Invent. Math. \textbf{70} (1982/83), no.~2, 169--218.

\end{thebibliography}
 \bibliographystyle{amsplain}

\end{document}